\documentclass[english]{amsart}

\usepackage{babel}
\usepackage{csquotes}

\usepackage{relsize}

\usepackage{hyperref}
\hypersetup{
  colorlinks=true,
  linkcolor=blue,
  filecolor=magenta,
  urlcolor=cyan,
}

\newtheorem{theorem}{Theorem}
\newtheorem{remark}{Remark}

\newtheorem{proposition}{Proposition}

\newtheorem{lemma}{Lemma}

\allowdisplaybreaks 
\DeclareMathOperator{\tr}{tr} 
\DeclareMathOperator{\Div}{div} 
\DeclareMathOperator{\vol}{Vol} 

\newcommand{\curvature}{\mathcal R}
\newcommand{\tube}{T}
\newcommand{\tangent}{\mathcal T}

\newcommand{\betafun}{\mathrm B}
\newcommand{\iiff}{\mathcal B}

\title[Minimal tubes of the sphere]{Minimal tubes of the sphere as the dimension grows: Volume and explosion of the mean exit time}
\author{Vicent Gimeno i Garcia}
\address{Department of Mathematics, Universitat Jaume I-IMAC, E-12071, Castell\'{o}, Spain}
\email{gimenov@uji.es}
\author{Martín Torres-Valverde}
\address{Department of Mathematics, Universitat Jaume I-IMAC, E-12071, Castell\'{o}, Spain}
\email{valverdm@uji.es}
\thanks{Research partially supported by the Research Program of Universitat Jaume I Project GACUJIMB/2025/11,  by the project AICO/2023/035 funded by Conselleria de Educaci\'o, Cultura, Universitats i Ocupaci\'o-Generalitat Valenciana,  and by predoctoral program of the Universitat Jaume I PREDOC/2025/80.}



\begin{document}

\begin{abstract}
The mean exit time of a Brownian motion
from a tube around a closed, minimal embedded
\(n\)-dimen\-sional submanifold of the \((n+m)\)-dimen\-sional
sphere is studied,
establishing two-sided bounds via radial transplant.
In accordance with known
results when \(n = 0\) or \(m = 1,\)
the mean exit time goes to infinity as \(n \to \infty\)
and to \(0\) as \(m \to \infty.\) Simultaneously, bounds for
the volume of such tubes are derived and compared to Weyl's
tube formula.
\end{abstract}

\maketitle
 
\section{Introduction}

Let \(M\) be a complete Riemannian manifold.
For a given region \(\Omega \subseteq M,\) consider a particle following
a standard Brownian motion starting at a point \(p \in \Omega.\)
Its \emph{exit time}
is the infimum time at which the particle escapes from \(\Omega,\) which
is a non-negative random variable (possibly infinite). Its expected value
defines a function
\[ E_\Omega : \Omega \to [0, \infty], \]
which is called the \emph{mean exit time} from \(\Omega.\) The region
\(\Omega\) is assumed to be bounded and not dense in \(M.\)
Then, as is well-known, \(E_\Omega\) is finite and can
be characterised as the unique
solution to the following Poisson problem with Dirichlet boundary
conditions, where we use the sign convention
\(\Delta f = \Div(\nabla f)\) for the Laplace operator:
\begin{equation}
\label{eq:poisson}
\left\{
\begin{aligned}
\Delta E_\Omega = -1 && \text{in the interior of \(\Omega\)}, \\
E_\Omega = 0 && \text{in \(\partial \Omega\).}
\end{aligned}
\right.
\end{equation}

The study of mean exit times is a traditional
problem in stochastic processes. In particular, a considerable
amount of research has been devoted to the exit time from
a region of a Riemannian manifold; see for instance
\cite{markvorsen1989mean,mcdonald13,dryden-langford17,markvorsen2006torsional,cadeddu15:maxim-exit-time,hurtado-markv-palmer12,gim-torsional-rigidity-flow}.
This problem is
also of interest in asymptotic geometric analysis, which concerns
itself with the evolution of geometric and analytic properties as the
dimension tends to infinity. The integral of the mean exit time function
is an object of independent interest, known as the
\emph{torsional rigidity}; see
\cite{polya48} or, more recently, \cite{kraus-simanek24}.

From  now on, we will restrict our attention to the
tubular regions of the sphere that are detailed next.
Let \(\Sigma^n\) be a closed (that is, compact, without boundary)
\(n\)-dimensional embedded submanifold of the \((n+m)\)-dimensional
round sphere \(M = \mathbb S^{n+m}.\) Let \(r : \mathbb S^{n+m} \to \mathbb R\)
be the (intrinsic) distance function \(r(p) = \operatorname{dist}(\Sigma, p).\)
Then the \emph{tube of radius \(R\) around \(\Sigma\)} is the region
\[ \Omega = \tube_R(\Sigma) = \{ p \in \mathbb S^{n+m}\ |\ r(p) \le R \}. \]
In particular, when \(n = 0,\) the tube around
a point is simply the geodesic ball of the same radius. 
We will assume that the tube does not self-intersect and write
\(E_R\) for \(E_\Omega.\)

The study of the mean exit time from a tubular region of the sphere
was undertaken in \cite{bessa2025concentrationmeanexittimes}.
When \(\Sigma = \mathbb S^n \hookrightarrow \mathbb S^{n+1}\)
is a totally geodesic hypersurface (in other words, an equator),
the following bound was obtained:
\begin{equation}
  \label{eq:cota-bessa-vicentes}
  E_R(p) \ge \frac{R-r(p)}{\sin R} \frac{1-\cos^{n+1} r(p)}{(n+1) \cos^{n} r(p)},
\end{equation}
where the right-hand-side approaches infinity as \(n\) tends to infinity,
a result reminiscent of the so-called \emph{fat equator effect}.
On the other hand, when  \(\Sigma = \{q\} \hookrightarrow \mathbb S^m\)
is a point, the following bound was obtained:
\begin{equation}
\label{eq:cota-bessa-vicentes-bola}
E_R(p) \le \frac{1-\cos R}{m\cos R},
\end{equation}
where the right-hand side approaches zero as \(m\) tends to infinity.

It was shown that the two bounds and the asymptotic behaviour above
generalize to minimal, closed embeddings
in ambient manifolds of Ricci curvature larger than that of the sphere.
However, the mean exit time from
tubes around submanifolds of any dimension and codimension remains unstudied.

In this paper, we establish bounds for the mean exit time from a tube
around a closed, minimal embedded submanifold of the sphere of arbitrary dimension
and codimension. The asymptotic behaviour of our lower bound when
the dimension \(n\) tends to infinity is the same
as in \eqref{eq:cota-bessa-vicentes}, with an extra factor of
\(n^{-(m-1)/2}\) that makes it diverge more slowly. The asymptotic behaviour of
our upper bound is the same as in \eqref{eq:cota-bessa-vicentes-bola}
when the codimension \(m\) tends to infinity, but in general the
inequality requires control over the second fundamental form.
\begin{theorem}
\label{thm:main}
    Let \(\Sigma^n\) be a minimal, closed embedded submanifold of
    the sphere \(\mathbb S^{n+m},\) with \(n \ge 1\) and \(m \ge 1.\)
    Fix \(0 < R < \pi/2\) such that the tube \(\tube_R(\Sigma)\)
    does not self-intersect. The mean exit time from \(\tube_R(\Sigma)\)
    admits:
    \begin{itemize}
        \item The lower bound
        \[
        E_R(p) \ge \frac{R - r(p)}{m\sin^m R} \sin^{m-1} \left(
        \sqrt{\frac2{\pi(n+1)}} \frac{m-1}m r(p) 
        \right) \frac{1 - \cos^{n+1} r(p)}{(n+1) \cos^n r(p)},
        \]
        where the right-hand side approaches infinity as \(n\) tends to infinity
        if \(R, r, m\) are fixed.
        \item Let \(c\) be the maximum principal curvature, and let
        \(A = \min_{\Sigma} (1-c\tan R) > 0.\)
        Assume the norm of the second fundamental form \(\iiff\) satisfies
        \begin{equation}
            \label{eq:bound-B-iiff}
            \|\iiff\|^2 \le \frac{Aq\cos^{n+1} R}{\tan^2 R}
        \end{equation}
        in all \(\Sigma.\) Then we have the upper bound
        \[ E_R(p) \le \frac1{m\cos^n RP_{\Sigma}(R)}
        \log\left(\frac{\cos r(p)}{\cos R}\right), \]
        where \(P_{\Sigma}(R) = 1-\tan^{2} R \max_{\Sigma}\|\iiff\|^2 /
        (Aq\cos^{n+1} R).\)
        The right-hand side approaches zero as \(m\) tends to infinity if \(R, r, n\)
        and a bound on \(\|\iiff\|^2\) are fixed.
    \end{itemize}
\end{theorem}
\begin{remark}
    The definition of \(A\) is made possible by
    Lemma 8.20 of \cite{gray90:_tubes}, which ensures that
    \(c\tan R < 1\) whenever \(R\) is small enough so that
    the tube does not self-intersect (which, as we will see, amounts to the
    condition that \(R\) is less than the \emph{minimal focal distance} of
    the embedding), and the sectional curvature
    of the ambient manifold satisfies \(\sec \ge 1.\)
\end{remark}
\begin{remark}
    Minimal, closed submanifolds of the sphere
    with bounded second fundamental form have
    been extensively studied. In \cite{an-min92:rigiditytheorem}, it was
    shown that the bound \(\|\iiff\|^2 \le (2/3)n\) implies that the submanifold
    is totally geodesic, with the exception of Veronese surfaces in
    \(\mathbb S^4.\) However, for large codimension, our condition
    \eqref{eq:bound-B-iiff} is not tight, and so the result
    may apply to submanifolds that fall outside of this classification.
\end{remark}
A shortcoming when we permit arbitrary codimension is that the behaviour does
not generalize easily to spaces with controlled Ricci curvature. Our upper
bound is also not immediately generalized to manifolds with bounded sectional
curvature.

Our proof of \autoref{thm:main} uses the method of
\cite{bessa2025concentrationmeanexittimes}. First, the exact solution
for the mean exit time in the case of a totally geodesic embedding
is calculated. The result is given by a certain function
\(f_R : \mathbb R \to \mathbb R\) composed with the distance function \(r\):
\begin{proposition}
    \label{thm:geodesic-embeddings}
    Let \(\mathbb S^n \hookrightarrow \mathbb S^{n+m}\)
    be a totally geodesic embedding of spheres, with \(m \ge 1.\)
    The mean exit time is the radial
    function \(E_R(p) = f_R(r(p))\) given by:
    \[ f_R(r) = \int_r^R \frac{\int_0^u \sin^{m-1} s \cos^n s\, ds}
    {\sin^{m-1} u \cos^n u}\, du,\]
    which satisfies \(f_R''(r) + f_R'(r) [(m-1) \cot r - n \tan r] = -1,\)
    with \(f_R(\pm R) = 0\) and \(f_R'(0) = 0.\)
\end{proposition}
\begin{remark}
    If we rescale to the sphere \(\mathbb S^{n+m}(\rho)\) of radius \(\rho\)
    (sectional curvature \(\frac1{\rho^2}\)), then the expression
    for the mean exit time must be rescaled somehow. Notice that it
    would be valid (after replacing \(R,\ r\)  by \(R/\rho,\ r/\rho\)
    respectively) for
    a Brownian motion with variance \(\rho^2\) per unit of time. Thus,
    if \(\sigma^2\) is the variance per unit of time,
    the function for the mean exit time becomes:
    \[ f_{R;\rho,\sigma^2}(r) = \frac{\rho^2}{\sigma^2}
    \int_{r/\rho}^{R/\rho} \frac{\int_0^u \sin^{m-1} s \cos^n s\, ds}
    {\sin^{m-1} u \cos^n u}\, du. \]
\end{remark}
Then, the mean exit time \(E_R\) from a tube around
an arbitrary closed submanifold is compared to the function \(f_R \circ r;\)
that is, a radial transplant. Lastly, some inequalities are obtained
for the function \(f_R.\)

As an application, we also obtain bounds for the volume of these tubes.
We use the divergence theorem for this, given that previous steps
provide a comparison between the Laplace operators of \(E_R\)
and \(f_R \circ r.\) In this way, we can state:
\begin{theorem}
\label{thm:volumes}
    Let \(\Sigma^n \hookrightarrow \mathbb S^{n+m}\) be a minimal, closed
    embedding (\(m, n \ge 1\)). Let
    \(\mathbb S^n \hookrightarrow \mathbb S^{n+m}\) be a totally
    geodesic embedding. Let \(0 < r < R.\)
    \begin{itemize}
        \item The tube volumes are related by:
    \[
    \frac{\vol(\tube_R(\Sigma))}{\vol(\tube_r(\Sigma))} \le
    \frac{\vol(\tube_R(\mathbb S^n))}{\vol(\tube_r(\mathbb S^n))}. \]
    As a consequence:
    \begin{equation}
    \label{eq:volume-inequality}
    \begin{aligned}
    \vol(\tube_R(\Sigma))
    &\le \frac{\vol(\tube_R(\mathbb S^n))}{\vol(\mathbb S^n)} \vol(\Sigma) \\
    &= \frac{2\pi^{m/2}}{\Gamma(m/2)} \left(\int_0^R \sin^{m-1} s \cos^n s\, ds \right) \vol(\Sigma).
    \end{aligned}
    \end{equation}
    \item
    Let \(\|\iiff\|^2,\  A\) and \(P_{\Sigma}(R)\) be as in \autoref{thm:main},
    and assume the bound \eqref{eq:bound-B-iiff} on the second
    fundamental form. The tube volumes satisfy:
    \[
    \frac{\vol(\tube_R(\Sigma))}{\vol(\tube_r(\Sigma))}
    \ge
    \left( \frac{\vol(\tube_R(\mathbb S^n))}{\vol(\tube_r(\mathbb S^n))}
    \right)^{P_{\Sigma}(R)}.
    \]
    \end{itemize}
\end{theorem}
\begin{remark}
    There currently exist estimations for the volume of a closed,
    minimal embedded
    submanifold \(\Sigma^n \hookrightarrow \mathbb S^{n+m}.\)
    One of them is, simply, \(\vol(\Sigma) \ge \vol(\mathbb S^n).\)
    A more refined one from \cite{chengli81:heat-equations-minimal},
    in the case where \(\Sigma\) is of maximal dimension (which means that
    \(\Sigma\) is not contained in a hyperplane of \(\mathbb R^{n+m+1}\)) is
    \[ \vol(\Sigma) > \left( 1 + \frac{2m+1}{B_n} \right) \vol(\mathbb S^n), \]
    where \(B_n\) is a constant that depends only on \(n.\)
\end{remark}
\begin{remark}
    Consider a
    tube of radius \(R\) around a closed embedded submanifold \(\Sigma^n\)
    in a manifold \(M^{n+m}\) of positive sectional curvature
    (\(\sec_M > 0\)). According
    to theorem 8.4 of \cite{gray90:_tubes}, the volume of the tube
    in \(M\) is bounded by:
    \begin{equation}
        \label{eq:weyl-inequality}
    \vol(\tube_R(\Sigma)) \le \frac{2\pi^{m/2} R^m}{m\Gamma(m/2)}
    \left(\vol(\Sigma) + \sum_{c=1}^{\lfloor{m/2}\rfloor}
    \frac{k_{2c}(\curvature^\Sigma - \curvature^M)R^{2c}}{(m+2)\cdots (m+2c)} \right),
    \end{equation}
    where the \(k_{2c}\) terms are known as \emph{integrated mean curvatures}.
    According to Corollary 9.25 of the same book, \(\vol(\tube_R(\Sigma))\) is
    smaller than the analogue expression with \(\curvature^M = 0,\) i.e., when
    \(\Sigma\) is embedded in \(\mathbb R^{n+m}\)
    (which implies \eqref{eq:weyl-inequality} is an
    equality, called \emph{Weyl's tube formula}).
    There is also
    a formula for the volume of the tube in \(M\) as a power expansion:
    \[ \vol(\tube_R(\Sigma)) = \frac{2\pi^{m/2} R^m}{m\Gamma(m/2)}
    \int_\Sigma \left(1 + O(R^2) \right)\, d\Sigma. \]
    (See Theorem 9.23 of the same book.)
    
    Our inequality \eqref{eq:volume-inequality} is a specialization to the
    spherical case. While of a similar spirit, it does without the summation
    or series expansion.
\end{remark}

\section*{Outline}

\begin{itemize}
    \item In \autoref{sec:preliminaries}, we
    review some fundamental concepts about tubes and establish some
    notation. Our exposition is based on that of \cite{gray90:_tubes}.
    \item In \autoref{sec:proof-met-totally-geodesic}, we prove 
    \autoref{thm:geodesic-embeddings}, deriving the function \(f_R\)
    that gives the mean exit time in the case of totally geodesic embeddings.
    \item In \autoref{sec:proof-main}, we establish \autoref{thm:main} about
    the mean exit time from minimal submanifolds.
    \item In \autoref{sec:proof-volumes}, we establish \autoref{thm:volumes}
    about the volume of minimal submanifolds.
\end{itemize}

\section{Preliminaries}
\label{sec:preliminaries}

Let \(\Sigma^n \hookrightarrow M^{n+m}\)
be an (isometric) embedding of the \(n\)-dimensional Riemannian manifold
\(\Sigma\)  in the \((n+m)\)-dimensional Riemannian manifold
\((M, \langle \cdot, \cdot \rangle)\) (both connected).
Assume furthermore that \(M\) is complete and \(\Sigma\) is closed.
Let \(\tangent \Sigma\) denote the tangent bundle of \(\Sigma.\)
The normal bundle \(\tangent \Sigma^\perp\) consists of all vectors that are tangent
to \(M\) and orthogonal to \(\Sigma.\) Let \(\exp: \tangent M \to M\) be the
exponential map of \(M,\) and let
\begin{align*}
    \exp_{\tangent \Sigma^\perp} : \tangent \Sigma^\perp &\to M \\
    (p, v) &\mapsto \exp_p(v)
\end{align*}
be the exponential map of the normal bundle. Let \(C(\Sigma)\) be the 
minimal focal
distance of \(\Sigma,\) which is the maximum length such that all
geodesics meeting \(\Sigma\) orthogonally are distance-minimizing.
By compactness, \(C(\Sigma) > 0.\)
For \(R > 0\) we consider the region
\[ \tube_R(\Sigma) = \{ \exp_{\tangent \Sigma^\perp}(p, v)\ |\ (p, v)
\in \tangent \Sigma^\perp,\ \|v\| \le R \}, \]
which is called the \emph{tube of radius \(R\) around \(\Sigma\)}.
We will henceforth assume that \(R\) is small enough so that
\[ \exp_{\tangent \Sigma^\perp} : \{(p, v) \in \tangent \Sigma^\perp\ |\ \|v\|
\le R \} \to \tube_R(\Sigma) \]
is a diffeomorphism, which amounts to the condition \(R < C(\Sigma).\) When this
happens, we say that
\(T_R(\Sigma)\) \emph{does not self-intersect}.

Since \(\Sigma\) has no boundary, the tube of radius \(R\) consists simply
of all the points of \(M\) whose distance to \(\Sigma\) does not exceed \(R,\)
which is the characterization we gave in the Introduction.
Likewise, the boundary \(\partial \tube_R(\Sigma)\) of the tube consists of
the points of \(M\) whose distance to \(\Sigma\) is exactly \(R.\) Thus,
a particle starting close to \(\Sigma\) will exit the tube precisely when
its distance to \(\Sigma\) reaches \(R.\)

Let \(N = \nabla r.\)
Each level set of the distance function \(r\) is a
hypersurface
\(\Sigma_r = \partial \tube_r(\Sigma).\) We define its \emph{shape operator}
or \emph{Weingarten map}
\begin{align*}
S: \tangent\Sigma_r &\to \tangent\Sigma_r, \\
S X &= - \nabla_X N.
\end{align*}
Note that, for \(X, Y\) tangent to \(\Sigma_r\):
\begin{align*}
\langle S X, Y \rangle
&= -\langle \nabla_X N, Y\rangle \\
&= \langle N, \nabla_X Y \rangle \\
&= \langle N, \nabla_Y X \rangle \\
&= -\langle \nabla_Y N, X \rangle \\
&= \langle X, S Y \rangle,
\end{align*}
where the third equality follows from \(\nabla_X Y - \nabla_Y X = [X, Y]\)
being tangent to \(\Sigma_r.\)
Thus, the shape operator is a self-adjoint \((1,\!1)\) tensor.
The \emph{second fundamental form} is its symmetric \((2,\!0)\)
version \(\iiff,\) defined by
\(\iiff(X, Y) = \langle S(X), Y \rangle. \) Note that \(S X\) is
tangent to \(\Sigma_r\) for any field \(X \in \mathfrak X(M),\) because
\(\|N\| = 1.\) The trace of the
shape operator is the \emph{mean curvature} \(H = \tr S = - \Delta r.\)
\begin{lemma}[Lemma 3.2 of \cite{gray90:_tubes}]
\label{thm:ricatti-shape-op}
The shape operator at \(\Sigma_r\) satisfies:
   \[ \nabla_{N} S = S^2 + \curvature(N, \cdot)N, \]
where \(\curvature\) is the curvature tensor, and \(S^2 X = S(SX).\)
\end{lemma}

The shape operator is defined for hypersurfaces and requires a choice
of unit normal vector (in this case, the outward normal \(N\))
to specify its sign.
If the submanifold \(\Sigma\) is of greater codimension then
there are more ways to choose a normal vector
\((p, v) \in \tangent\Sigma^\perp\) with \(\|v\| = 1.\) Each of these
gives rise to a \((1,\!1)\) tensor \(S_v\) on \(\tangent_p \Sigma\),
satisfying:
\[ \langle S_v(X), Y \rangle = \langle \nabla_X Y, v \rangle \]
This is the \emph{Weingarten map} of \(\Sigma,\) and its \((2,\!0)\)
form is the \emph{second fundamental form} \(\mathcal B_v.\)

Let \(\gamma\) be a unit-speed geodesic emanating orthogonally
from \(\Sigma.\)
The Generalized Gauss Theorem of \cite[p.~26]{gray90:_tubes}
asserts that \(\gamma'(r) = N = \nabla r.\)
(We use \(r\) for the curve parameter because
the distance from \(\gamma(r)\) to \(\Sigma\) is exactly \(r\) whenever
\(r \le R.\))

The \emph{principal curvatures} are the eigenvalues
\(c_1, \dots, c_n\) of the Weingarten map \(S_{\gamma'(0)},\)
which may be extended to eigenvalues
\(\kappa_1(r), \dots, \widehat{\kappa_{n+1}}(r), \dots,
\kappa_{n+m}(r)\) of the shape operator along \(\gamma,\)
with associated orthonormal eigenvectors
\( F_1(r), \dots,\allowbreak F_{n+m}(r), \) with the convention
\(F_{n+1}(r)\allowbreak = N = \gamma'(r).\)
These satisfy a Ricatti differential equation derived from
\eqref{thm:ricatti-shape-op}:
\begin{lemma}[Corollary 3.5 of \cite{gray90:_tubes}]
\label{thm:ricatti-principal-curv}
    If \(F_i\) is differentiable at \(\gamma(r),\) then:
    \[ \kappa_i'(r) = \kappa_i(r)^2 +
    \curvature(F_i, N, F_i, N). \]
\end{lemma}
The last term in the lemma above
is simply the sectional curvature of the plane generated
by \(N\) and \(F_i.\)
Since we are going to take the ambient manifold to be a sphere,
we may choose the
eigenvectors to be parallel (see Lemma 6.13 and Theorem 6.14 of
\cite{gray90:_tubes}),
and in particular differentiable.
We have chosen the principal curvatures of \(\Sigma_r\)
so that \(\lim_{r \to 0} \kappa_i(r) = c_i\) are the principal
curvatures of \(\Sigma\) in the direction of \(\gamma'(0)\)
for \(i = 1, \dots, n.\) The other principal curvatures satisfy:
\begin{lemma}
  \label{thm:minus-infty}
  \(\lim_{r\to0} \kappa_i(r) = -\infty\) for \(i = n+2, \dots, n+m.\)
\end{lemma}
\begin{proof}
  Our proof of this already-known fact
  rests on Fermi coordinates, a generalization of normal
  coordinates which is explained in detail in
  \cite[ch.~2]{gray90:_tubes}. A system of Fermi coordinates
  \(x^1, \dots, x^{n+m}\) for \(T_R(\Sigma)\)
  around \(p \in \Sigma\) can be defined when we have
  a coordinate system \(y^1, \dots, y^n\) in an open set \(U \subseteq \Sigma\)
  around \(p,\) and orthonormal sections \(E_{1}, \dots, E_{m}\) of
  \(\tangent\Sigma^\perp.\) An arbitrary point of \(\tangent U^\perp\)
  can be written as \(\sum_{j = 1}^m t^j E_j(q)\) for a certain
  \(q \in U.\) The associated Fermi coordinate system on
  \(T_R(U) = \{\exp_{\tangent\Sigma^\perp}(q, v)\ |\ (q, v) \in \tangent U^\perp\ |\ \|v\| \le R\}\)
  is given by:
  \begin{align*}
      x^i \left( \exp_{\tangent \Sigma^\perp}
      \sum_{j = 1}^m t^{j}E_{j}(q) \right) &= y^i(q), && i = 1, \dots, n; \\
      x^{n+i} \left( \exp_{\tangent \Sigma^\perp}
      \sum_{j = 1}^m t^{j}E_{j}(q)\right) &= t^i, && i = 1, \dots, m.
  \end{align*}
  Note that the coordinate fields
  \(\partial_1, \dots, \partial_n\) are tangent to \(\Sigma\) at \(p,\)
  whereas \(\partial_{n+1}, \dots,\allowbreak \partial_{n+m} \) are orthogonal to
  \(\Sigma\) at \(p.\)
  
  Let's take a Fermi coordinate system \(x^1, \dots, x^{n+m}\)
  around \(p = \gamma(0) \in \Sigma\) that also satisfies 
  \(\partial_{n+1}|_{\gamma(r)} = \gamma'(r) = N_{\gamma(r)} \)
  and \(\partial_j|_{\gamma(0)} = F_j\) for \(j=1, \dots, n.\)
  Take one of the eigenvectors
  \(F \in \{F_{n+2}, \dots, F_{n+m} \}\)
  with associated eigenvalue \(\kappa.\)
  We will show that \(\lim_{r\to0} r\kappa(r) = -1,\) from
  which the lemma follows. Writing \(F = F^i \partial_i\)
  (using the Einstein summation convention), we have:
  \begin{align*}
      r\kappa(r)
      &= r\langle S F, F\rangle_{\gamma(r)} \\
      &= r\langle -\nabla_{F} N, F \rangle_{\gamma(r)} \\
      &= -r\langle F^j \nabla_{\partial_j} N, F \rangle_{\gamma(r)} \\
      &= -r\langle F^j (\nabla_{N} \partial j + [\partial_j, N]), F \rangle_{\gamma(r)} \\
      &= -rF^j \langle \nabla_{\partial_{n+1}} \partial_j, F \rangle_{\gamma(r)}
          -r\langle F^j [\partial_j, N], F \rangle_{\gamma(r)}.
      \end{align*}
    The first term tends to zero as \(r \to 0.\)
    The Lie Brackets in the second term are calculated
    in Lemma 2.8 of \cite{gray90:_tubes}.
    They are \([\partial_j, N] = 0\)
    for \(j = 1, \dots, n+1,\) whereas their value for
    \(j = n+2, \dots, n+m\) is:
    \[ [\partial_j, N] = \frac{1}{r}\partial_j - \frac{\partial_j r}{r} N. \]
    Thus:
    \begin{align*}
    \lim_{r \to 0} r\kappa(r)
    &= \lim_{r\to0} -r\langle F^j [\partial_j, N], F \rangle_{\gamma(r)} \\
    &= \lim_{r\to0} r\left\langle - \sum_{j=n+2}^{n+m} \frac{F^j\partial_j}{r}, F \right\rangle_{\gamma(r)} \\
    &= \lim_{r \to 0} -\left\langle \sum_{j=n+2}^{n+m} F^j \partial_j, F \right\rangle_{\gamma(r)} \\
    &= \lim_{r \to 0} -\left\langle \sum_{j=1}^{n+m} F^j \partial_j, F \right\rangle_{\gamma(r)} \\
    &= \lim_{r\to0} -\langle F, F \rangle_{\gamma(r)} \\
    &= -1.
    \end{align*}
    The fourth equality requires some explanation. First,
    \(\langle \partial_{n+1}, F \rangle_{\gamma(r)} =
    \langle N, F \rangle_{\gamma(r)} = 0.\)
    Second, for \(j = 1, \dots, n,\) we have that
    \(\partial_j|_{\gamma(0)} = F_j|_{\gamma(0)}.\) Hence,
    \(\lim_{r\to0}\langle \partial_j, F \rangle_{\gamma(r)} \allowbreak=
    \lim_{r\to0} \langle F_j, F \rangle_{\gamma(r)} = 0.\)
    Lastly, each of the \(F^j\) are bounded because \(\|F\| = 1.\)
\end{proof}

\section{Proof of \autoref{thm:geodesic-embeddings}}
\label{sec:proof-met-totally-geodesic}
Let \(\varphi : \mathbb S^n \hookrightarrow \mathbb S^{n+m}\) be the
totally geodesic embedding.
Fix \(0 < R < \pi/2;\) then the exponential map from
\(\{(p, v) \in (\tangent\mathbb S^n)^\perp\ |\ \|v\| \le R \}\) to the tube
\(\tube_R(\varphi(\mathbb S^n))\) is a diffeomorphism.

Because of the symmetry of the sphere, the solution of
\eqref{eq:poisson} must be radial. That is, \[E_R(p) = f_R(r(p))\]
for a certain function \(f_R : \mathbb R \to \mathbb R.\) For this reason,
we can express the Laplace operator of \(E_R\) as:
\begin{align*}
    \Delta E_R = \Div(\nabla (f_R \circ r)) &= \Div(f_R'(r) N) \\
                                            &= f_R''(r) \|N\|^2 + f_R'(r) \Delta r \\
                                            &= f_R''(r) - f_R'(r) \tr(S) \\
                                            &= f_R''(r) - f_R'(r) (\kappa_1 + \dots + \kappa_{n+m}).
\end{align*}

Choose the orthonormal basis of eigenvectors \(F_i\)
and the principal curvatures \(c_i\) and \(\kappa_i\) along
a geodesic as above.
Now \autoref{thm:ricatti-principal-curv} reads like this:
\begin{align*}
\kappa_{i}'(r) &= \kappa_i(r)^2 + \curvature_{\gamma'(r)F_i(r)\gamma'(r)F_i(r)} \\
               &= \kappa_i(r)^2 + 1.
\end{align*}
The solution to this equation is
analytic:
\[ \kappa_i(r) = \tan(r+\arctan(c_i)),\]
where \(c_i = \lim_{r \to 0} \kappa_i(r).\) are the corresponding
principal curvatures of \(\Sigma.\) In particular, \(c_1, \dots, c_n\)
are the eigenvalues of the Weingarten map, which vanish
because the embedding is totally geodesic. From \autoref{thm:minus-infty},
we have
\(c_{n+2} = \dots = c_{n+m} = -\infty.\)
This allows us to rewrite \eqref{eq:poisson} as:
\[ f_R''(r) - f_R'(r) (-n \tan r + (m-1) \cot r) = -1. \]

  If we take \(f_R'\) as the unknown, this is just a
  first-order linear ODE , which may be solved by
  the standard method of finding an integrating factor.
  The symmetry condition
  \(f_R'(0) = 0\)
  can be used to determine the constant of integration. Thus we have:
  \[ f_R'(r) = - \frac{\int_0^r \sin^{m-1} s \cos^n s\, ds}{\sin^{m-1} r \cos^n r}. \]
  We integrate one more time, and use the boundary condition
  \(f_R(R) = 0.\) This leads to the solution:
  \[ f_R(r) = \int_r^R \frac{\int_0^u \sin^{m-1} s\cos^n s\, ds}{\sin^{m-1} u \cos^n u}\, du, \]
  as desired.

\section{Proof of \autoref{thm:main} (mean exit time)}
\label{sec:proof-main}
Suppose we have a closed manifold \(\Sigma^n\) embedded in \(\mathbb S^{n+m}\)
(\(m, n \ge 1\)).
Given \((p, u) \in \tangent\Sigma^\perp\) with \(\|u\| = 1,\)
the eigenvalues of the second
fundamental form \(\iiff_u\) are the principal curvatures \(c_1, \dots, c_n;\)
in particular, the minimality condition means that the mean curvature
\(H_u = c_1 + \dots + c_n = 0.\) The norm of the second
fundamental form is defined as \(\|\iiff_u\|^2 = c_1^2 + \dots + c_n^2.\)

Fix a number \(0 < R < C(\Sigma),\) where \(C(\Sigma)\)  is the minimal
focal distance. We will study the solution
\(E_R : \tube_R(\Sigma) \to \mathbb R\)
to the Poisson problem \eqref{eq:poisson}, where \(\tube_R(\Sigma)\) is the
tube of radius \(R\) around \(\Sigma.\)
We will split the proof into two parts:
\begin{enumerate}
    \item First, we will radially transplant the function
    \(f_R = f_R^{(n, m)}\) defined in
    \autoref{thm:geodesic-embeddings} to our embedding. By comparing
    the Laplacian of the exact solution \(E_R\) to the Laplacian
    of the radial function \(f_R \circ r,\) we will show that:
    \[ E_R(p) \ge f_R(r(p)). \]
    In the other direction, let
    \(A = 1-\tan R\max_{i,\Sigma} c_i > 0,\)
    and let
    \[P(R) = 1 - \frac{\tan^2 R \max_{\Sigma} \|\iiff\|^2}{Aq\cos^{n+1} R}. \]
    Then, assuming the second fundamental form satisfies \eqref{eq:bound-B-iiff},
    we will show:
    \[ E_R(p) \le \frac1{P(R)} f_R(r(p)) =
    \frac{Aq\cos^{n+1} R}{Aq\cos^{n+1} R -
    \tan^2 R \max_{\Sigma}\|\iiff\|^2}f_R(r(p)). \]
    \item Next, we will get the lower bound
    \[ f_R^{(n, m)}(r) \ge \frac{R-r}{m\sin^m R} \sin^{m-1} \left(
    \sqrt{\frac2{(n+1)\pi}} \frac{m-1}m r \right)
    \frac{1-\cos^{n+1} r}{(n+1)\cos^n r}, \]
    which implies \(f_R^{(n,m)}(r) \to \infty\) as \(n \to \infty,\)
    and the upper bound
    \[ f_R^{(n,m)}(r) \le \frac{1}{m\cos^{n-1} R}
    \log\left(\frac{\cos r}{\cos R}\right), \]
    which implies \(f_R^{(n,m)}(r) \to 0\) as \(m \to \infty.\)
\end{enumerate}
The theorem follows immediately once we chain the corresponding inequalities
from the first and the second part.

\subsection*{First part: radial transplant}
Consider the radial function \(\tilde E_R(p) = f_R(r(p))\) on \(\tube_R(\Sigma).\)
Its Laplacian is given by:
\[ \Delta \tilde E_R(p) = f_R'(r(p)) (n\tan r(p) - (m-1)\cot r(p) - \tr(S(p))) - 1. \]
Here, as happened with the geodesic spheres, the shape operator has
eigenvalues \(\kappa_i(r) = \tan(r + \arctan(c_i)),\) where
\(c_i = -\infty\) for \(i = n+2, \dots, n+m.\) In this way:
\[ \Delta \tilde E_R = f_R'(r) \left[ n \tan r -
\sum_{i = 1}^n \tan(r + \arctan(c_i)) \right] - 1. \]
Note that \(1 - c_i\tan r \ge A > 0\) for all \(i = 1, \dots, n,\) whereas
\( f_R'(r) \le 0. \)
Using the formula for the tangent of the sum of two angles and the first few
terms of the geometric series of ratio \(c_i \tan r\):
\begin{align*}
    &\Delta \tilde E_R \\
    =& f_R'(r) \left[ n \tan r - \sum_{i=1}^n \frac{\tan r + c_i}{1 - c_i \tan r} \right] - 1 \\
    =& f_R'(r) \left[ n \tan r - \sum_{i = 1}^n \left( \frac{\tan r}{1 - c_i \tan r} + \frac{c_i}{1-c_i \tan r} \right) \right] - 1 \\
    =& f_R'(r) \left[
    n\tan r - \mathlarger{\mathlarger{{\sum_{i = 1}^n}}} \left( 
    \begin{gathered}
    \tan r \left( 1 + c_i\tan r +  \frac{c_i^2\tan^2 r}{1 - c_i\tan r} \right) \\
    + c_i \left( 1 + \frac{c_i \tan r}{1 - c_i \tan r} \right)
    \end{gathered} \right)
    \right] - 1 \\
    =& f_R'(r) \left[-\mathlarger{\mathlarger{\sum_{i=1}^n}}
    \left( 
    \begin{gathered}
     c_i\tan^2 r + c_i
    + \tan^2 r \frac{c_i^2 \tan r}{1-c_i\tan r} + \frac{c_i^2 \tan r}{1-c_i\tan r}
    \end{gathered}  \right) \right] - 1\\
    =& -f_R'(r) (1+\tan^2 r) \sum_{i=1}^n \left( c_i + \frac{c_i^2 \tan r}{1-c_i\tan r} \right) -1 \\
    =& -\frac{f_R'(r)}{\cos^2 r}\left( H + \sum_{i=1}^n \frac{c_i^2 \tan r}{1-c_i\tan r} \right) - 1 \\
    =& -\frac{f_R'(r)}{\cos^2 r}\sum_{i=1}^n \frac{c_i^2 \tan r}{1-c_i\tan r} -1 \\
    \ge& -1.
\end{align*}

On the other hand, the exact solution of \eqref{eq:poisson} satisfies
\(\Delta E_R = -1,\) and both \(E_R\) and \(\tilde E_R\) vanish at the
boundary of the tube (where \(r = R\)). Hence:
\begin{equation}
\label{eq:subharmonic-1}
  \Delta (\tilde E_R - E_R) \ge 0,
\end{equation}
that is, \(\tilde E_R - E_R\) is a sub-harmonic function that vanishes
at the boundary of the tube. By the maximum principle,
\(\tilde E_R - E_R \le 0,\) or \(E_R \ge f_R \circ r,\) as desired.

For the second inequality, note that
\begin{align*}
    -f_R'(r) &= \frac{\int_0^r \sin^{m-1} s \cos^n s\, ds}{\sin^{m-1} r \cos^n r} \\
    &\le \frac{\int_0^r \sin^{m-1} s \cos s\, ds}{\sin^{m-1} r \cos^n r} \\
    &= \frac{\sin^m r}{m\sin^{m-1} r \cos^n r} \\
    &= \frac{\sin r}{m\cos^n r},
\end{align*}
from which we obtain:
\begin{align*}
    \Delta \tilde E_R &= - \frac{f_R'(r) \tan r}{\cos^2 r} \sum_{i=1}^n \frac{c_i^2}{1 - c_i \tan r} - 1 \\
    &\le - \frac{f_R'(r) \tan r}{\cos^2 r} \sum_{i=1}^n \frac{c_i^2}{A} - 1 \\
    &= -\frac{f_R'(r) \tan r \|\iiff\|^2}{A\cos^2 R} - 1 \\
    &\le \frac{\sin r \tan r \max_{\Sigma}\|\iiff\|^2}{Aq\cos^{n+2} R} - 1 \\
    &\le \frac{\tan^2 R \max_{\Sigma}\|\iiff\|^2}{Aq\cos^{n+1} R} - 1 \\
    &= -\frac{Aq\cos^{n+1} R - \tan^2 R \max_{\Sigma} \|\iiff\|^2}{Aq\cos^{n+1} R} \\
    &= -P_{\Sigma}(R).
\end{align*}
As a consequence, as long as \(P(R) \ge 0,\) which is condition \eqref{eq:bound-B-iiff}:
\begin{equation}
    \label{eq:subharmonic-2}
    \Delta E_R - \frac1{P(R)} \Delta \tilde E_R \ge 0.
\end{equation}
This time it's the function \(E_R - \frac1{P(R)} \tilde E_R\) that
is sub-harmonic and vanishes on the boundary of the tube, so
\[E_R \le \frac{1}{P(R)} \tilde E_R = \frac{Aq\cos^{n+1} R -
\tan^2 R \max_{\Sigma} \|\iiff\|^2}{Aq\cos^{n+1} R} (f_R \circ r), \]
which is what we wanted to show.

\subsection*{Second part: bounds for totally geodesic embeddings}
We're going to bound the function \(f_R\) from
\autoref{thm:geodesic-embeddings}. As a preliminary step, we obtain
exact formulas for certain definite trigonometric integrals, using the
change of variables \(\sin^2 r = u\):
\begin{equation}
\label{eq:int-trig-beta}
    \begin{aligned}
    \int_0^{\pi/2} \sin^{a} r \cos^b r\, dr
    &= \frac12 \int_0^1 u^{(a-1)/2} (1-u)^{(b-1)/2}\, du \\
    &= \frac12\betafun((a+1)/2, (b+1)/2) \\
    &= \frac{\Gamma((a+1)/2)\Gamma((b+1)/2)}{2\Gamma((a+b+2)/2)},
    \end{aligned}
\end{equation}
where \(\betafun\) and \(\Gamma\) are the beta and gamma functions, respectively.
For integers \(a\) and \(b,\) the formula above may also be obtained
inductively.
We also invoke an approximation for the Gamma function. It's a double inequality
by Wendel \cite{qi-luo13}: whenever \(x > 0,\) for each
\(0 \le s \le 1,\) we have:
\begin{equation*}
    x^{1-s} \le \frac{\Gamma(x+1)}{\Gamma(x+s)} \le (x+s)^{1-s}.
\end{equation*}
When applied to \eqref{eq:int-trig-beta}, for \(a = 0,\ b \ge 1/2,\) we obtain:
\begin{equation}
\label{eq:gautschi}
\sqrt{\frac{\pi}{2(b+1)}} \le \int_0^{\pi/2} \cos^b s\, ds \le
\sqrt{\frac\pi{2b}}.
\end{equation}

We are ready to obtain the bounds for \(f_R(r) = f_R^{(n,m)}(r),\)
where \(0 \le r \le \pi/2\) and \(n \ge 1,\ m \ge 1.\) We
go for the lower bound first. We're interested in the
asymptotic behaviour as \(n \to \infty,\) and so for the moment
we will consider the rest of the parameters to be fixed. That is,
we shall not make all dependencies explicit at all times, except
for the dependency on \(n.\) To begin with, we restate
the bound \eqref{eq:cota-bessa-vicentes} from
\cite{bessa2025concentrationmeanexittimes}
using our notation:
\begin{equation}
\label{eq:cota-bessa-vicentes-con-f}
   f_R^{(n,1)} \ge \frac{R - r}{\sin R} \frac{1 - \cos^{n+1} r}{(n+1)\cos^n r}.
\end{equation}
Note now that the function
    \[
    u \mapsto \frac{\int_0^u \sin^{m-1} s \cos^n s\, ds}{\int_0^u \cos^n s\, ds}
    \]
defined on the interval \(r \le u \le R\)
is increasing, and so it takes its minimum value at \(u = r.\) Hence:
\begin{align*}
    f_R^{(n, m)}(r) &= \int_r^R \frac{\int_0^u \sin^{m-1} s \cos^n s\, ds}{\sin^{m-1} u \cos^n u}\, du \\
    &= \int_r^R \frac{\int_0^u \sin^{m-1} s \cos^n s\, ds}{\int_0^u \cos^n s\, ds} \frac{\int_0^u \cos^n s\, ds}{\sin^{m-1} u\cos^n u}\, du \\
    &\ge \frac1{\sin^{m-1} R}\frac{\int_0^r \sin^{m-1} s\cos^n s\, ds}{\int_0^r \cos^n s\, ds} f_R^{(n,1)}(r).
\end{align*}
We may regard the integral \(\int_0^u \sin^{m-1} s \cos^n s\, ds\)
as the integral of \(\sin^{m-1} s\) with respect to the measure
\(d\mu_n(s) = \cos^n s\, ds.\)
Let
    \[\varepsilon_n = \sin^{m-1} \left( \frac\delta{\sqrt{n+1}} \right),\]
    for \(\delta \ge 0.\) By Markov's inequality
    \cite[p.~91]{stein-shakarchi05}:
    \begin{align*}
    f_R^{(n, m)}(r) &\ge \frac1{\sin^{m-1} R}
    \frac{\varepsilon_n \int_{\delta/\sqrt{n+1}}^r
    \cos^n s\, ds}{\int_0^r \cos^n s\, ds} f_R^{(n,1)}(r) \\
    &= \frac1{\sin^{m-1} R} \varepsilon_n \left( 1 - \frac{\int_0^{\delta/\sqrt{n+1}} \cos^n s\, ds}{\int_0^r \cos^n s\, ds} \right) f_R^{(n,1)}(r) \\
    &\ge \frac1{\sin^{m-1} R} \varepsilon_n \left( 1 - \frac{\delta/\sqrt{n+1}}{\int_0^r \cos^n s\, ds} \right) f_R^{(n,1)}(r).
    \end{align*}
Since the function \(r \mapsto (\int_0^r \cos^n s\, ds)/r\) defined
on \(0 \le r \le \pi/2\) is decreasing, it attains its minimum value
at \(r = \pi/2,\) from which
\(\int_0^r \cos^n s\, ds \ge (2r/\pi) \int_0^{\pi/2} \cos^n s\, ds. \)
Recalling Wendel's inequality \eqref{eq:gautschi}:
    \begin{align*}
    f_R^{(n, m)}(r)&\ge \frac1{\sin^{m-1} R} \varepsilon_n \left( 1 - \frac{\delta/\sqrt{n+1}}{(2r/\pi)\int_0^{\pi/2} \cos^n s\, ds} \right) f_R^{(n,1)}(r) \\
    &\ge \frac{1}{\sin^{m-1} R}\varepsilon_n \left( 1 - \frac{\delta/\sqrt{n+1}}
    {(2r/\pi)\sqrt{\pi/(2n+2)}} \right) f_R^{(n, 1)}(r) \\
    &= \frac{1}{\sin^{m-1} R} \varepsilon_n \left( 1 - \frac\delta r
    \sqrt{\frac{\pi }{2}} \right) f_R^{(n, 1)}(r).
    \end{align*}
Recalling \eqref{eq:cota-bessa-vicentes}:
    \begin{align*}
        f_R^{(n, m)}(r) \ge \frac{R-r}{\sin^m R} \varepsilon_n\left( 1 - \frac\delta r \sqrt{\frac{\pi}{2}} \right) \frac{1-\cos^{n+1} r}{(n+1)\cos^{n} r}.
    \end{align*}
Using the value \(\delta = r \sqrt{2/\pi}\, (m-1)/m,\)
which is close to optimal for large \(n,\)
we arrive at the desired lower bound:
    \begin{align*}
        f_R^{(n, m)}(r) &\ge \frac{R-r}{m\sin^m R}
        \sin^{m-1}\left( \sqrt{\frac2{\pi (n+1)}} \frac{m-1}m r \right) \frac{1-\cos^{n+1} r}{(n+1)\cos^{n} r}.
    \end{align*}
    In particular, \(f_R^{(n, m)}(r) \to \infty\) as \(n \to \infty,\)
    if \(m, R, r\) are fixed.

Let's finish by obtaining an upper bound for \(f^{(n, m)}_R(r),\) 
where again \(0 \le r \le R \le \pi/2\) and \(n \ge 1,\ m \ge 1\):
    \begin{align*}
        f^{(n,m)}_R(r) &= \int_r^R \frac{\int_0^u \sin^{m-1} s \cos^n s\, ds}{\sin^{m-1} u \cos^n u}\, du \\
        &\le \frac1{\cos^{n-1} R} \int_r^R\frac{\int_0^u \sin^{m-1} s\cos s\, ds}{\sin^{m-1} u \cos u}\, du \\
        &\le \frac{1}{m\cos^{n-1} R} \int_r^R \tan u\, du \\
        &= \frac{1}{m\cos^{n-1} R} \log \left( \frac{\cos r}{\cos R} \right).
    \end{align*}
In particular, \(f_R^{(n, m)}(r) \to 0\) as \(m \to 0,\) if the other
parameters remain fixed.

\section{Proof of \autoref{thm:volumes} (volume)}
\label{sec:proof-volumes}
Let \(\mathbb S^{n+m}\) be the ambient sphere, where we consider the embedding
of a closed, minimal submanifold \(\Sigma^n,\) as well as the embedding of
the totally geodesic submanifold of the same dimension \(\mathbb S^n.\)
Let \(V(R) = \vol(\tube_R(\Sigma))\) (resp. \(\tilde V(R) = \vol(\tube_R(\mathbb S^n))\))
be the volume of the tube of radius \(R\) around \(\Sigma\) (resp. \(\mathbb S^n\)).

Let \(\omega\) be the Riemannian volume form of \(\mathbb S^n,\) and let
\(\omega_{\tangent\Sigma^{\perp}}\) be the Riemannian volume form of \(\tangent\Sigma^{\perp},\)
with the signs chosen in such a way that \(\exp_{\tangent\Sigma^{\perp}}\) preserves
the orientation.
Fix \((p, v) \in \tangent\Sigma^{\perp}\) with \(\|v\| = 1.\)
The \emph{infinitesimal change of volume function} \(\vartheta\)
\cite[ch.~3]{gray90:_tubes}
is characterised by the equality:
\[ \vartheta_v(r)\omega_{\tangent\Sigma^{\perp}} = \exp_{\tangent\Sigma^\perp}^*(\omega) \]
at each point \((p, rv) \in \tangent\Sigma^\perp,\) where \(r < C(\Sigma)\) (minimal
focal distance). That is, \(\vartheta_v(r)\) is the factor by which the volume
of a small cube around \((p, rv)\) gets multiplied when applying
\(\exp_{\tangent\Sigma^\perp}.\) Since \(\vartheta_v(0) = 1\) and we are restricting
our attention to a region
where \(\exp_{\tangent\Sigma^\perp}\) is regular, it must always be the case
that \(\vartheta_v(r) > 0.\) The volume of the tube can now be written as:
\[ V(R) = \int_0^R \int_\Sigma \int_{S^{m-1}} r^{m-1} \vartheta_v(r)\, du\, d\Sigma\, dr. \]

Let \(E_R\) denote the mean exit time function of \(\tube_R(\Sigma)\)
and \(\tilde E_R = f_R \circ r\)
the radial transplant previously defined.
From \eqref{eq:subharmonic-1}, we know that \(\Delta E_r \le \Delta \tilde E_r.\)
Using the divergence theorem:
\begin{align*}
    V(r) &= \int_{\tube_r(\Sigma)} 1\, dV \\
    &= \int_{\tube_r(\Sigma)} - \Delta E_r\, dV \\
    &\ge \int_{\tube_r(\Sigma)} - \Delta \tilde E_r\, dV \\
    &= \int_{\partial \tube_r(\Sigma)} - \langle \nabla \tilde E_r, \nabla r \rangle\, dA \\
    &= \int_{\partial \tube_r(\Sigma)} - f_r'(r) \langle \nabla r, \nabla r \rangle\, dA \\
    &= -f_r'(r)\int_{\partial \tube_r(\Sigma)} dA \\
    &= -f_r'(r) \mathrm{Area}(\partial \tube_r(\Sigma)) \\
    &= -f_r'(r) V'(r).
\end{align*}
(See \cite[p.~42]{gray90:_tubes} for the last equality.) The inequality becomes
an equality in the totally geodesic case. Thus:
\begin{equation}
\label{eq:volume-diff-inequality}
     \frac{V'(r)}{V(r)} \le -\frac{1}{f_r'(r)},
\end{equation}
and
\[ \frac{\tilde V'(r)}{\tilde V(r)} = -\frac{1}{f_r'(r)} =
\frac{\cos^{n} r\sin^{m-1} r}{\int_0^r \cos^n s \sin^{m-1} s\, ds}. \]
As an aside, this latter equality can be integrated to obtain the volume of
a tube around a totally geodesic submanifold of the sphere:
\begin{align*}
\tilde V(r) &= \frac{2\vol(\mathbb S^{n+m})}{\betafun((n+1)/2, m/2)}
\int_0^r \sin^{m-1} s \cos^n s\, ds \\
&= \frac{4\pi^{(n+m+1)/2}}{\Gamma((n+1)/2) \Gamma(m/2)}
\int_0^r \sin^{m-1} s \cos^n s\, ds.
\end{align*}
From \eqref{eq:volume-diff-inequality} we also get that
\[ \frac{d}{d r} \log V(r) = \frac{V'(r)}{V(r)}
\le -\frac{1}{f_r'(r)} = \frac{\tilde V'(r)}{\tilde V(r)}
= \frac{d}{d r} \log \tilde V(r), \]
and so \((d/d r) \log(V(r) / \tilde V(r)) \le 0;\) that is,
the function \(r \mapsto V(r) / \tilde V(r)\) is decreasing. In other
words, whenever \(0 < r < R,\) we have:
\[ \frac{V(R)}{\tilde V(R)} \le \frac{V(r)}{\tilde V(r)}, \]
which is equivalent to the first desired inequality.

Now we take limits as \(r\) approaches \(0\):
\begin{align*}
    \frac{V(R)}{\tilde V(R)} &\le \lim_{r\to0} \frac{V(r)}{\tilde V(r)} \\
    &= \lim_{r \to 0} \frac{\int_0^r \int_{\Sigma}\int_{S^{m-1}}
    s^{m-1} \vartheta_u(s)\, du\, d\Sigma\, ds}
    {\int_0^r \int_{S^n}\int_{S^{m-1}}
    s^{m-1} \tilde \vartheta_u(s)\, du\, d\mathbb S^n\, ds} \\
    &= \lim_{r\to0} \frac{\int_{\Sigma}
    \int_{S^{m-1}} r^{m-1} \vartheta_u(r)\, du\, d\Sigma}{
    \int_{S^n} \int_{S^{m-1}} r^{m-1} \tilde \vartheta_u(r)
    \, du\, d\mathbb S^n} \\
    &= \frac{\int_{\Sigma} \int_{S^{m-1}} \vartheta_u(0)\, du\, d\Sigma}{
    \int_{S^n} \int_{S^{m-1}} \tilde \vartheta_u(0)\, du\, dS^n
    } \\
    &= \frac{\int_{\Sigma} \int_{S^{m-1}} 1\, du\, d\Sigma}{
    \int_{S^n} \int_{S^{m-1}} 1\, du\, dS^n}\\
    &= \frac{\vol(\Sigma)\vol(\mathbb S^{m-1})}{\vol(\mathbb S^n)
    \vol(\mathbb S^{m-1})}.
\end{align*}
L'Hôpital's rule is permitted in the third line because both
the numerator and the denominator both approach \(0\) and the
limit of the quotient of the derivatives exists.
Rearranging, we obtain the second stated inequality.

The approach used for the other bound is completely analogous, only using the
Laplace operator inequality \eqref{eq:subharmonic-2} instead of \eqref{eq:subharmonic-1}:
\begin{align*}
    V(r) &= \int_{\tube_r(\Sigma)} 1\, dV \\
    &= \int_{\tube_r(\Sigma)} -\Delta E_r\, dV \\
    &\le \int_{\tube_r(\Sigma)} -\frac1{P(r)} \Delta \tilde E_r\, dV \\
    &= \frac1{P(r)} \int_{\partial \tube_r(\Sigma)} -\langle \nabla \tilde E_r, \nabla r \rangle\, dA \\
    &= \frac1{P(r)} \int_{\partial \tube_r(\Sigma)} -f_r'(r) \langle \nabla r, \nabla r \rangle\, dA \\
    &= -\frac{f_r'(r)}{P(r)}\int_{\partial \tube_r(\Sigma)} dA \\
    &= -\frac{f_r'(r)}{P(r)}\mathrm{Area}(\partial \tube_r(\Sigma)) \\
    &= -\frac{f_r'(r)}{P(r)}V'(r).
\end{align*}
Consequently, for \(R >r\):
\[
\frac{V'(r)}{V(r)} \ge -\frac{P(r)}{f_r'(r)} =
P(r)\frac{\tilde V'(r)}{\tilde V(r)} \ge P(R)\frac{\tilde V'(r)}{\tilde V(r)}.
\]
Thus,
\[ \frac{d}{d r} \log \frac{V(r)}{\tilde V(r)^{P(R)}} \ge 0; \]
in other words, the function \(r \mapsto \frac{V(r)}{\tilde V(r)^{P(R)}}\)
is increasing, and so we arrive at the final result:
\[ \frac{V(R)}{V(r)} \ge \left( \frac{\tilde V(R)}{\tilde V(r)} \right)^{P(R)}. \]

\textbf{Artificial Intelligence Acknowledgement.} This paper did not use
generative artificial intelligence to generate the \LaTeX{} code, the
content or the results.

\providecommand{\bysame}{\leavevmode\hbox to3em{\hrulefill}\thinspace}
\providecommand{\MR}{\relax\ifhmode\unskip\space\fi MR }
\providecommand{\MRhref}[2]{%
  \href{http://www.ams.org/mathscinet-getitem?mr=#1}{#2}
}
\providecommand{\href}[2]{#2}

\end{document}